\documentclass{article}

\usepackage[english]{babel}

\usepackage[letterpaper,top=2cm,bottom=2cm,left=3cm,right=3cm,marginparwidth=1.75cm]{geometry}

\usepackage{amssymb,mathrsfs,amscd}
\usepackage{amsthm}
\usepackage{hyperref}
\usepackage{amsmath}
\usepackage{lineno,hyperref,color}
\usepackage{amsmath}
\usepackage{graphicx}

\title{Li-Yorke chaos for maps on $G$-Space}
\author{Yingcui Zhao}

\begin{document}
\newtheorem{theorem}{Theorem}[section]
\newtheorem{corollary}[theorem]{Corollary}
\newtheorem{lemma}[theorem]{Lemma}
\newtheorem{proposition}[theorem]{Proposition}
\newtheorem{problem}[theorem]{Problem}
\newtheorem{maintheorem}[theorem]{Main Theorem}
\newtheorem{definition}[theorem]{Definition}
\newtheorem{remark}[theorem]{Remark}
\newtheorem{example}[theorem]{Example}
\newtheorem{claim}{Claim}[section]
\maketitle

\begin{abstract}
 We introduce the definition of Li-Yorke chaos for the map $f$ on $G$-spaces, and show $G$-Li-Yorke chaos is iterable for $f$. Li-Yorke chaos implies $G$-Li-Yorke chaos, while the converse is not true. Then we give a sufficient condition for $f$ to be chaotic in the sense of $G$-Li-Yorke. Also, we prove that if $f$ is $G$-transitive and there exists a common fixed point for $f$ and all of the maps in $G$, then $f$ is chaotic in the sense of $G$-Li-Yorke.
\end{abstract}

\section{Introduction}
Chaos comes from practical problems such as physics, biology, economics, medicine and
chemistry. The theory of chaotic dynamical system is one of the challenging issue of dynamical systems. By convention, ``chaos'' means ``a state of confusion without order". While
in the theory of chaotic dynamical systems, ``chaos'' is a technical term. 
The term ``chaos'' in discrete dynamical systems was introduced firstly by Li and Yorke in 1975 \cite{LiYorke}. 

Let $(X,d)$ be a metric space and $f: X\rightarrow X$ be a continuous map. $Y\subset X$ is said to be a Li-Yorke chaotic set of $f$, if it contains at least two points and any pair of distinct points $x, y\in Y$ is a Li-Yorke pair, i. e.,
$$\liminf_{n\rightarrow\infty} d(f^{n}(x), f^{n}(y))=0,\ \limsup_{n\rightarrow\infty} d(f^{n}(x), f^{n}(y))>0.$$
Moreover, if there exists an uncountable Li-Yorke chaotic set $S\subset X$ of $f$, $f$ is said to be chaotic in the sense of Li-Yorke.

In the last fifty years, a large number of papers have been devoted to the study of chaos in the sense of Li-Yorke. Xiong\cite{G} showed that if $f$ is transitive and there exists a fixed point for $f$, then $f$ is chaotic in the sense of Li-Yorke. Wang\cite{Wang2013} gave a sufficient condition for $f$ to be chaotic in the strong sense of Li-Yorke. Recently, Li-Yorke chaos remains a research hotspot in the field of topological dynamical systems, as detailed in \cite{LY1,LY2,LY3,LY4}. 

In addition to the topology domain, chaos theory is also being researched in other settings. Recently, chaos is being defined and studied for group actions, set-valued maps, iterated funtion system and so on. The present paper focuses on $G$-space.

$G$-spaces, in the field of topological dynamical systems, provides a framework for studying the dynamics of actions of a group on a topological space. By investigating $G$-spaces, researchers can gain insights into the behavior and properties of dynamical systems under group actions, leading to a deeper understanding of the interplay between topology and dynamics in various contexts. Additionally, the study of G-spaces can help uncover new phenomena, reveal connections between different areas of mathematics, and contribute to the development of theoretical foundations for applications in fields such as physics, engineering, and biology. In recent years, an increasing number of scholars have started to pay attention to the study of $G$-spaces. Ali and Ekta defined and studied the notion of chain transitivity for maps on G-spaces in \cite{Ali2019}. Phinao and Khundrakpam investigated the $G$-mixing, $G$-sensitive and $G$-shadowing property of $f$ in \cite{Phinao2021}. Raad and Iftichar studied the sequence $G$-asymptotic average shadowing with $G$-chain transitivity in \cite{Raad2020}. For more references, please see \cite{G1,G2}.

The aim of the present paper is to define and study Li-Yorke chaos for maps on $G$-spaces. 
The specific layout of the present paper is as follows. In Section 2, we introduce some preliminaries and definitions. In Section 3, we show $G$-Li-Yorke chaos is iterable for equivariant maps. In Section 4, We study the relationship between Li-Yorke and $G$-Li-Yorke chaos. In Section 5. We give a sufficient condition for $G$-Li-Yorke chaos. In Section 6, We give a theorem which can be used to identify whether a map is $G$-Li-Yorke chaotic or not. In Section 7, we summarize the conclusions.
\section{Preliminaries and Basic Concepts}
\label{Sec:2}
Throughout this paper, let $(X,d)$ be a metric space and $f$ be a continuous map from X to X. Let $\mathbf{N}=\{0,1,2,\cdots\}$, $\mathbf{Z^+}=\{1,2,3,\cdots\}$, $\mathbf{Z}=\{\cdots,-2,-1,0,1,2,\cdots\}$.
Let $G$ be a topological group and $\varphi$ be a continuous map from $G\times X$ to $X$\cite{Ahmadi2014}. The $(X,G,\varphi)$ or $X$ is said to be a metric $G$-space if the following conditions hold:
\begin{itemize}
	\item[(1)] $\varphi(e,x)=x$ for all $x\in X$ where $e$ is the identity of $G$.
	\item[(2)]$\varphi(g_1,\varphi(g_2,x))=\varphi(g_1g_2,x)$ for all $x\in X$ and all $g_1,g_2\in G$.
\end{itemize}
If $X$ is compact, then $X$ is said to be compact metric $G$-space. For the convenience of writing, $\varphi(g,x)$ is usually abbreviated as $gx$.
For $x\in X$, the $G-orbit$ of $x$, denoted by $G_f(x)$, is given as the set $\{gf^k(x)|g\in G,k\geq0\}$.

The notion of equivariant map was introduced in \cite{Ekta2013}. We recall the definition.
\begin{definition}
	Let $(X,d)$ be a metric $G$-space. We say $f$ is an \emph{equivariant map}, if $f(gx)=gf(x)$ for all $x\in X$ and all $g\in G$.
\end{definition}
In \cite{Trandefi}, the notion of topological transitivity on $G$-space was defined and studied. We recall the definition.
\begin{definition}
	Let $(X,d)$ be a metric $G$-space. $f$ is said to be \emph{$G$-transitive}, if for any non-empty open sets $U,V\subset X$, there exist $n\in\mathbf{Z^+}$ and $g\in G$ such that $gf^n(U)\bigcap V\neq\emptyset$.
\end{definition}
Naturally, we define the notion of $G$-transitive point of $f$ as follows.
\begin{definition}
	Let $(X,d)$ be a metric $G$-space. $x\in X$ is said to be a \emph{$G$-transitive point} of $f$, if $\{gf^k(x)|g\in G,k\in\mathbf{N}\}$ is dense in $X$.
\end{definition}

Now we define the notion of $G$-recurrent point and $G$-Li-Yorke chaos for maps on $G$-spaces.
\begin{definition}
	Let $(X,d)$ be a metric $G$-space. A point $x\in X$ is said to be a \emph{$G$-recurrent point} of $f$, if there exist sequence $\{n_i\}_{i=1}^\infty$ and $\{g_i\}_{i=1}^\infty\subset G$ such that $\lim_{i\rightarrow\infty}g_if^{n_i}(x)=x.$ The set of all $G$-recurrent points of $f$ is denoted by $R_G(f).$
\end{definition}

\begin{definition}
	Let $(X,d)$ be a metric $G$-space. $Y\subset X$ is said to be a \emph{$G$-Li-Yorke chaotic set} of $f$, if it contains at least two points and any pair of distinct points $x, y\in Y$ is a Li-Yorke pair, i. e., there exists $\{g_n\}_{n=1}^\infty\subset G$ such that
	$$\liminf_{n\rightarrow\infty} d(g_nf^{n}(x), g_nf^{n}(y))=0,\ \limsup_{n\rightarrow\infty} d(g_nf^{n}(x), g_nf^{n}(y))>0.$$
	Moreover, if there exists an uncountable $G$-Li-Yorke chaotic set $S\subset X$ of $f$, $f$ is said to be chaotic in the sense of $G$-Li-Yorke.
\end{definition}

\section{Iteration of $G$-Li-Yorke chaos}
Before we show the iteration of $G$-Li-Yorke chaos for equivariant maps, we give a well known lemma.
%

\begin{lemma}\label{lem1}
	Let $\{n_{i}\}$ be a sequence of positive integers. For any $n>0$,
	there exist $r$, $0\leq r<n$, a subsequence $\{n_{i_{j}}\}$ of
	$\{n_{i}\}$ and a sequence $\{q_{j}\}$ of non-negative integers such
	that $n_{i_{j}}=nq_{j}+r$.
\end{lemma}

\begin{theorem}\label{prop1}
	Let $(X,d)$ be a metric $G$-space. Suppose that $f$ is an equivariant map. $f$ is chaotic in the sense of $G$-Li-Yorke if and only if for any $N\in\mathbf{Z^+}$, $f^N$ is chaotic in the sense of $G$-Li-Yorke.
\end{theorem}
\begin{proof}
	Let $N=1$, sufficiency is obvious. Now we show the necessity.
	
	Let $N\in\mathbf{Z^+}$. Suppose that $f$ is chaotic in the sense of $G$-Li-Yorke. Then there exists an uncountable $G$-Li-Yorke chaotic set $S\subset X$ of $f$. For any given distinct points $x, y\in S$, there exists $\{g_n\}_{n=1}^\infty\subset G$ such that
	$$\liminf_{n\rightarrow\infty} d(g_nf^{n}(x), g_nf^{n}(y))=0,\ \limsup_{n\rightarrow\infty} d(g_nf^{n}(x), g_nf^{n}(y))>0.$$
	\begin{itemize}
		\item[(1)] By $\liminf_{n\rightarrow\infty} d(g_nf^{n}(x), g_nf^{n}(y))=0$, there exist $x_0\in X$ and an infinite sequence $\{n_i\}$ such that $\lim_{n_i\rightarrow\infty} g_{n_i}f^{n_i}(x)=x_0$ and $\lim_{n_i\rightarrow\infty} g_{n_i}f^{n_i}(y)=x_0$. By Lemma \ref{lem1}, there exist $r, 0\leq r<N$ and a sequence $\{q_{i}\}$
		such that
		$$\lim_{{q_{i}}\rightarrow\infty}g_{Nq_i+r}(f^{N})^{q_{i}}(f^{r}(x))=x_{0}, ~~~\text{and}~~~ \lim_{{q_{i}}\rightarrow\infty}g_{Nq_i+r}(f^{N})^{q_{i}}(f^{r}(y))=x_{0}.$$
		Let $l=N-r$. Since $f$ is an equivariant map, as $i\rightarrow\infty$,
		$$f^{l}(g_{Nq_i+r}(f^{N})^{q_{i}}(f^{r}(x)))=g_{Nq_i+r}(f^{N})^{q_{i}+1}(x)\rightarrow
		f^l(x_{0}),$$
		$$f^{l}(g_{Nq_i+r}(f^{N})^{q_{i}}(f^{r}(y)))=g_{Nq_i+r}(f^{N})^{q_{i}+1}(y)\rightarrow
		f^l(x_{0}).$$
		
		\item[(2)] By $\limsup_{n\rightarrow\infty} d(g_nf^{n}(x), g_nf^{n}(y))>0$, there exist distinct points $a, b\in X$ and an infinite sequence $\{n_i\}$ such that $\lim_{n_i\rightarrow\infty} g_{n_i}f^{n_i}(x)=a$ and $\lim_{n_i\rightarrow\infty} g_{n_i}f^{n_i}(y)=b$. By Lemma \ref{lem1}, there exist $r, 0\leq r<N$ and a sequence $\{q_{i}\}$
		such that
		$$\lim_{{q_{i}}\rightarrow\infty}g_{Nq_i+r}(f^{N})^{q_{i}}(f^{r}(x))=a ~~~\text{and}~~~ \lim_{{q_{i}}\rightarrow\infty}g_{Nq_i+r}(f^{N})^{q_{i}}(f^{r}(y))=b.$$ There exist a subsequence $\{q_{i_{j}}\}$ of $\{q_{i}\}$ and
		$a_{1},b_{1}\in X$ such that
		$$\lim\limits_{j\rightarrow\infty}g_{Nq_{i_j}}(f^{N})^{q_{i_{j}}}(x)=a_{1} ~~~\text{and}~~~ \lim\limits_{j\rightarrow\infty}g_{Nq_{i_j}}(f^{N})^{q_{i_{j}}}(y)=b_{1}.$$
		Since $f$ is an equivariant map, as $j\rightarrow\infty$,
		$$f^{r}g_{Nq_{i_j}}(f^{N})^{q_{i_{j}}}(x)=g_{Nq_{i_j}}(f^{N})^{q_{i_{j}}}(f^r(x))\rightarrow f^r(a_{1})=a,$$
		$$f^{r}g_{Nq_{i_j}}(f^{N})^{q_{i_{j}}}(y)=g_{Nq_{i_j}}(f^{N})^{q_{i_{j}}}(f^r(y))\rightarrow f^r(b_{1})=b.$$
		Then $a_{1}\neq b_{1}$.
	\end{itemize}
	By $(1)$ and $(2)$, there exists $\{h_n\}_{n=1}^\infty\subset G$ such that
	$$\liminf_{n\rightarrow\infty} d(h_nf^{n}(x), h_nf^{n}(y))=0,\ \limsup_{n\rightarrow\infty} d(h_nf^{n}(x), h_nf^{n}(y))>0.$$
	Hence, $(x,y)$ is a $G$-Li-Yorke pair of $f^N$. Then, $S$ is an uncountable $G$-Li-Yorke chaotic set of $f^N$. So, $f^N$ is chaotic in the sense of $G$-Li-Yorke.
\end{proof}
The following example shows the condition "$f$ is an equivariant map" in Proposition \ref{prop1} can't be removed.
\begin{example}
	Consider $f,g_1,g_2:[-2,2]\rightarrow[-2,2]$ defined as, $f(x)=-|x|$,
	\begin{equation} g_{1}(x)=
		\begin{cases}
			\frac{2}{3}x-\frac{2}{3}, & -2\leq x\leq -\frac{1}{2}, \nonumber\\
			2x, & -\frac{1}{2}< x\leq 0, \nonumber\\
			x, & 0<x\leq 2,
		\end{cases}
	\end{equation}
	\begin{equation} g_{2}(x)=
		\begin{cases}
			-2x-2, & -2\leq x\leq -\frac{1}{2}, \nonumber\\
			-\frac{2}{5}x-\frac{6}{5}, & -\frac{1}{2}<x\leq 2.
		\end{cases}
	\end{equation}
	Suppose that $G$ is generated by any combination and iteration of $g_1,g_2,\phi,g_1^{-1},g_2^{-1}$, where $\phi$ is an identify map, i. e. $\phi(x)=x, x\in [-2,2]$. Note that $g_1f(\frac{1}{2})=-1$, while $fg_1(\frac{1}{2})=-\frac{1}{2}$. So, $f$ is not an equivariant map. As we all know, tent map $h_1(x)=1-|1-2x|$, $x\in[0,1]$ is chaotic in the sense of Li-Yorke. Suppose that $h(x)=-x$, $x\in[-1,0]$ and
	\begin{equation} h_{2}(x)=
		\begin{cases}
			-2x-2, & -1\leq x\leq -\frac{1}{2}, \nonumber\\
			2x, & -\frac{1}{2}<x\leq 0.
		\end{cases}
	\end{equation}
	Then, $h$ is the topological conjugation from $h_1$ to $h_2$. Hence, $h_2$ is also chaotic in the sense of Li-Yorke. So, $[-1,0]$ is an uncountable $G$-Li-Yorke chaotic set of $f$. That is $f$ is chaotic in the sense of $G$-Li-Yorke.
	
	However, $f^2$ is not chaotic in the sense of $G$-Li-Yorke.
\end{example}
\section{Relationship between Li-Yorke and G-Li-yorke chaos}
Since $G$ is a group, it contains identity element. Then we have the following theorem.
\begin{theorem}\label{thm1}
	Let $(X,d)$ be a metric $G$-space. If $f$ is chaotic in the sense of Li-Yorke, then it is chaotic in the sense of $G$-Li-Yorke.
\end{theorem}
However, the converse of Theorem \ref{thm1} is not true.
\begin{example}
	Consider $f,g_1,g_2:[-2,2]\rightarrow[-2,2]$ defined as $f(x)=x$,
	\begin{equation} g_{1}(x)=
		\begin{cases}
			x, & -2\leq x\leq 0, \nonumber\\
			2x, & 0< x\leq \frac{1}{2}, \nonumber\\
			\frac{2}{3}x+\frac{2}{3}, & \frac{1}{2}<x\leq 2,
		\end{cases}
	\end{equation}
	\begin{equation} g_{2}(x)=
		\begin{cases}
			-\frac{2}{5}x+\frac{6}{5}, & -2\leq x\leq \frac{1}{2}, \nonumber\\
			-2x+2, & \frac{1}{2}<x\leq 2.
		\end{cases}
	\end{equation}
	It is easy to see $f$ is not chaotic in the sense of Li-Yorke. Suppose that $G$ is generated by any combination and iteration of $g_1,g_2,f,g_1^{-1},g_2^{-1}.$ Then, $[0,1]$ is an uncountable $G$-Li-Yorke chaotic set of $f$. Hence, $f$ is chaotic in the sense of $G$-Li-Yorke. So, Li-Yorke chaos $\nRightarrow$ $G$-Li-Yorke chaos.
\end{example}
\section{A sufficient condition for $G$-Li-Yorke chaos}
Now, we give a sufficient condition for $f$ to be chaotic in the sense of $G$-Li-Yorke.
\begin{proposition}
	Let $(X,d)$ be a metric $G$-space and $f$ be an equivariant map. Let $\{n_{i}\}_{i=1}^\infty$ be a sequence of
	positive integers and $\{A_{i}\}_{i=0}^\infty$ and
	$\{B_{i}\}_{i=0}^\infty$ be decreasing sequences of compact sets
	satisfying
	$$\bigcap^\infty_{i=0}A_{i}=\{a\}, ~~~ \bigcap^\infty_{i=0}B_{i}=\{b\},$$
	where $a,b\in X, a\neq b$. If for any $i\geq0$, there exists $g_i\in G$ such that $$A_{i+1}\bigcup B_{i+1}\subset g_{n_i}f^{n_i}(A_i)\bigcap g_{n_i}f^{n_i}(B_i),$$ then $f$ is chaotic in the sense of $G$-Li-Yorke.
\end{proposition}
\begin{proof}
	Let $$\mathcal{C}=\{c=C_0C_1\cdots|C_i\in\{A_i,B_i\},i=0,1,\cdots\},$$
	and $E\subset\Sigma_2=\{x=x_0x_1\cdots|x_i\in\{0,1\},i=0,1,\cdots\}$ be an uncountable set such
	that for any distinct points $s=s_{0}s_{1}\cdots,
	t=t_{0}t_{1}\cdots\in E, s_{n}=t_{n}$ for infinitely many $n$ and
	$s_{m}\neq t_{m}$ for infinitely many $m$. Denote $\varphi(x):E\rightarrow \mathcal{C}$ by $\varphi(x)=C_0C_1\cdots$, where
	\begin{equation} C_i=
		\begin{cases}
			A_i, & x_i=0, \nonumber\\
			B_i, & x_i=1,
		\end{cases}
	\end{equation}
	for any $i=0,1,2,\cdots$. Let $D=\varphi(E)$. It is easy to see $\varphi$ is injective. Since $E$ is uncountable, $\varphi(E)$ is uncountable. Let $c=C_0C_1\cdots\in D$ and $k>1$, then there exists $x_k\in C_0$ such that for every $i=0,1,\cdots,k-1$, $$g_{n_i}f^{n_i}g_{n_{i-1}}f^{n_{i-1}}\cdots g_{n_0}f^{n_0}(x_k)\in C_{i+1}.$$ Let $x_c$ be one of limit points of $\{x_k\}_{k=0}^\infty$ and $S=\{x_c|c\in D\}$, then $S$ is an uncountable $G$-Li-Yorke chaotic set of $f$. So, $f$ is chaotic in the sense of $G$-Li-Yorke.
\end{proof}
\section{Main Result}

Before presenting the main result, we introduce some definitions and list several lemmas playing an important role in the proof of the main result.

$x\in X$ is said to be a \emph{periodic point} of $f$, if there exists $n>0$ such that $f^n(x)=x$. The smallest such positive integer $n$ is said to be the \emph{period} of
$x$. And the periodic point $x$ with period $n$ is said to be a \emph{$n$-periodic point}. When $n=1$, $n$-periodic point $x$ is said to be a \emph{fixed point} of $f$. 
\begin{definition}
	$x\in X$ is said to be a \emph{$G$-recurrent point} of $f$, if there exist $\{g_n\}_{n=1}^\infty\subset G$ and a sequence $\{n_i\}\subset \mathbf{Z^+}$ such that $$\lim_{i\rightarrow\infty}g_if^{n_i}(x)=x.$$
	We use $R_G(f)$ to denote the set of the $G$-recurrent points of $f$.
\end{definition}
\begin{lemma}\label{lem2}\cite{G} Let $H$ be a family of real-valued functions from $X$ to $[0, +\infty)$. Denote $\mu: X\rightarrow[0, +\infty)$ as: for any $x\in X$, $$\mu(x)=inf\{h(x)| h\in H\}. $$If each $h\in H$ is semicontinuous, then $g$ is semicontinuous.
\end{lemma}

\begin{lemma}\label{lem3}\cite{G}
	Suppose that for each integer $i\geq 1$, $h_i: X\rightarrow [0, +\infty]$ is semicontinuous. Let $a\in [0, +\infty]$, $$g(x)=\liminf_{i\rightarrow\infty}h_i(x), x\in X. $$
	If the set $g^{-1}([0, a])$ is dense in $X$, then it is a dense $G_\delta$ set of $X$.
\end{lemma}

\begin{lemma}\label{lem4}
	Let $(X,d)$ be a metric $G$-space. Suppose that $f$ is an equivariant map. If there exist a $G$-transitive point of $f$ and a fixed point $\vartheta\in X$ of $f$ which satisfies $g(\vartheta)=\vartheta$ for any $g\in G$, then there exists a dense $G_\delta$ set $B$ in $X\times X$ such that for any $(x_1, x_2)\in B$, 
	$$\liminf_{n\rightarrow\infty} d(f^{n}(u), f^{n}(v))=0,\forall u\in G(x_1), \forall v\in G(x_2).$$
\end{lemma}
\begin{proof}
	Denote $F:X\times X\rightarrow [0,+\infty)$ as for any $(u,v)\in X\times X$, $$F(x,y)=\liminf_{n\rightarrow\infty} d(f^{n}(u), f^{n}(v)).$$
	Let $\omega\in X$ be a $G$-transitive point of $f$. Then there exist a sequence $\{n_j\}_{j=3}^\infty$ and $\{g_j\}_{j=3}^\infty\subset G$ such that $\lim_{j\rightarrow\infty}g_jf^{n_j}(\omega)=\vartheta$. For any $Z=(z_1,z_2)\in G_{f}(\omega)\times G_{f}(\omega)$,
	there exist $k_1, k_2\in \mathbf{N}$ and $g_1,g_2\in G$ such that $g_1f^{k_1}(\omega)=z_1,g_2f^{k_2}(\omega)=z_2.$ Then for $i=1,2$ $$\lim_{j\rightarrow\infty}g_jf^{n_j}(z_i)=\lim_{j\rightarrow\infty}g_jf^{n_j}(g_if^{k_i}(\omega))=\vartheta.$$
	
	Hence, $F(z)=0$. Since $G_{f}(\omega)\times G_f(\omega)$ is dense in $X\times X$, by Lemma \ref{lem3}, $F^{-1}(0)$ is a dense $G_\delta$ set in $X\times X$.
\end{proof}
\begin{lemma}\label{lem5}
	Let $(X,d)$ be a metric $G$-space and $f$ be a continuous map from $X$ to $X$. If the $G$-recurrent points set $R_G(f)$ is dense in $X$, then $R_G(f)$ is a dense $G_\delta$ set.
\end{lemma}
\begin{proof}
	Denote $F:X\rightarrow[0,+\infty)$ as for any $x\in X$, $$F(x)=\liminf_{n\rightarrow\infty}\inf_{g\in G}d(gf^n(x),x).$$
	Then $x\in R_G(f)\Leftrightarrow F(x)=0$. So, $F^{-1}(0)$ is dense in $X$. For any $n\geq1$ and any $g\in G$, $d_n(x):=d(gf^n(x),x)$ is semicontinuous. By Lemma \ref{lem2}, for any $n\geq1$, $h_n(x)=\inf_{g\in G}d_n(x)$ is semicontinuous. Then by Lemma \ref{lem3} $R_G(f)=F^{-1}(0)$ is a dense $G_\delta$ set in $X$.
\end{proof}
\begin{lemma}\label{lem6}
	Let $(X,d)$ be a metric $G$-space. If $f$ is $G$-transitive, then $R_G(f\times f)$  is a dense $G_\delta$ set in $X\times X$.
\end{lemma}
\begin{proof}
	Since $X$ is compact, there exist numerable topological basis $\{U_n\}_{n=1}^\infty$. By $$\{x\in X|\overline{G_f(x)}=X\}=\bigcap_{n=1}^\infty\bigcup_{m=0}^\infty\bigcup_{g\in G}gf^{-m}(U_n),$$
	and Proposition 2.4 of \cite{He2014}, the open set $\bigcup_{m=0}^\infty\bigcup_{g\in G}gf^{-m}(U_n)$ is dense in $X$. By Baire Theorem, $\{x\in X|\overline{G_f(x)}=X\}$ is dense in $X$. Then, there exists $\omega\in X$ being a $G$-transitive point of $f$. Hence, $\{gf^k(\omega)|g\in G, k\geq0\}$ is dense in $X$. Then, $$\{gf^k(\omega)|g\in G, k\geq0\}\times\{gf^k(\omega)|g\in G, k\geq0\}\subset R_G(f\times f).$$ So, $R_G(f\times f)$  is dense in $X\times X$. By Lemma \ref{lem5}, $R_G(f\times f)$  is a dense $G_\delta$ set in $X\times X$.
\end{proof}

\begin{theorem}\label{th1}
	Let $(X,d)$ be a metric $G$-space. Suppose that $f$ is an equivariant map. If $f$ is $G$-transitive and there exists a fixed point $\vartheta\in X$ of $f$ which satisfies $g(\vartheta)=\vartheta$ for any $g\in G$, then there exists a $\mathbf{c}$ dense $G$-Li-Yorke set in $X$ for $f$. Specifically, $f$ is chaotic in the sense of $G$-Li-Yorke.
\end{theorem}
\begin{proof}
	Let $D=R_G(f\times f)\bigcap B$, in which $B$ is the same with the $B$ in Lemma \ref{lem4}. Therefore, by Lemma \ref{lem4} and Lemma \ref{lem6} $D$ is a residual set in $X\times X$. Since $X$ is compact, it is complete and separable. Then by [2] there exists a dense Mycielski set $K$ such that for any distinct $x_1, x_2\in K, (x_1, x_2)\in D$.
	
	So, $K$ satisfies the condition of the present theorem.
\end{proof}

The following example shows the condition "a fixed point $\vartheta\in X$ of $f$ which satisfies $g(\vartheta)=\vartheta$ for any $g\in G$" in Theorem \ref{th1} can't be removed.

\begin{example}
	Suppose $\mathcal{S}^1$ is a unit circumference on complex plane. Denote $f$ as $f(z)=z$, $\forall z\in\mathcal{S}^1$, and $g(z)=ze^{2\pi i\lambda}$, $\forall z\in\mathcal{S}^1$, in which $\lambda$ is an irrational number. $G$ is generalized by $g$, i.e., $G=\{g^n|n\in\mathbf{Z}\}$, in which $g^0=f$. $f$ is an equivariant map. Since $g$ is transitive, $f$ is $G$-transitive. For any $z\in\mathcal{S}^1$, $g(z)\neq z$. Since $g$ is an isometric mapping, $f$ is not chaotic in the sense of $G$-Li-Yorke.
\end{example}

\begin{corollary}
	Let $(X,d)$ be a metric $G$-space. Suppose that $f$ is an equivariant map. If $f$ is $G$-transitive and there exists a $n$-periodic point $\varrho\in X$ of $f$ which satisfies $g(\varrho)=\varrho$ for any $g\in G$, then there exists a $\mathbf{c}$ dense $G$-Li-Yorke set in $X$ for $f$. Specifically, $f$ is chaotic in the sense of $G$-Li-Yorke.
\end{corollary}
\begin{proof}
	By $G$-transitivity of $f$, there exists $\omega\in X$ being a $G$-transitive point of $f$. Set $$X_0=\overline{\{gf^n(\omega)|g\in G, n\in\mathbf{N}\}}.$$ For any given nonempty subsets $U,V$ of $X_0$, there exist $g_1,g_2\in G$, and $n_1,n_2\in\mathbf{N}$ such that $g_1f^{n_1}(\omega)\in U$ and $g_2f^{n_2}(\omega)\in V$. Without loss of generality, let $n_1\geq n_2$. Then, $(g_1g_2^{-1}f^{n_1-n_2})(g_2f^{n_2}(\omega))=g_1f^{n_1}(\omega)\in U$. So, $$g_1g_2^{-1}f^{n_1-n_2}(V)\bigcap U\neq\emptyset.$$ That is, $f^n$ is $G$-transitive. By Theorem \ref{thm1} $f^n$ is chaotic in the sense of $G$-Li-Yorke. By Proposition \ref{prop1} $f$ is chaotic in the sense of $G$-Li-Yorke.
\end{proof}
\section{Conclusions}
We introduce the definition of Li-Yorke chaos for the map on $G$-spaces. And we show:
\begin{itemize}
	\item[(1)]If $f$ is equivariant, then its Li-Yorke chaos is iterable. And the condition that $f$ is equivariant can't be removed.
	\item[(2)]Li-Yorke chaos $\Rightarrow\nLeftarrow$ $G$-Li-Yorke chaos.
	\item[(3)]A sufficient condition for $f$ to be $G$-Li-Yorke chaotic is given.
\end{itemize}

The above conclusions are the generalization of the continuous maps of general compact metric space. They will enrich the theory of $G$-space and iterated funtion systems. It provides the theoretical basis and scientific foundation for the application of chaos in computational mathematics and biological mathematics.

\end{document}